\documentclass[12pt,reqno]{amsart}

\usepackage{graphicx,subfigure}
\usepackage{hyperref}
\hypersetup{colorlinks=true, citecolor=blue, linkcolor=red}

\numberwithin{equation}{section} \numberwithin{figure}{section}
\numberwithin{table}{section} 
{ \theoremstyle{plain}

}

{ \theoremstyle{definition}
\newtheorem{thm}{Theorem}

\newtheorem{rem}{Remark}

\numberwithin{equation}{section} \numberwithin{lem}{section}
\numberwithin{thm}{section} \numberwithin{cor}{section}
\numberwithin{pro}{section} \numberwithin{rem}{section}
}

\begin{document}

\title[On the existence of dark solitons of the defocusing cubic NLSE]{On the existence of dark solitons of the defocusing cubic nonlinear Schr\"{o}dinger equation with periodic inhomogeneous nonlinearity}



\author{Christos Sourdis} \address{Department of Mathematics and Applied Mathematics, University of
Crete.}
              \email{csourdis@tem.uoc.gr}           




\maketitle

\begin{abstract}
We provide a simple proof of the  existence of dark solitons of
the defocusing cubic nonlinear Schr\"{o}dinger equation with
periodic inhomogeneous nonlinearity. Moreover, our proof allows
for a broader class of inhomogeneities and gives some new
properties of the solutions. We also apply our approach to the
defocusing cubic-quintic nonlinear Schr\"{o}dinger equation with a
periodic potential.
\end{abstract}

We consider the following defocusing cubic nonlinear
Schr\"{o}dinger equation with inhomogeneous nonlinearity on
$\mathbb{R}$:
\begin{equation}\label{eqNLS}
i\psi_t=-\frac{1}{2}\psi_{xx}+g(x)|\psi|^2\psi,
\end{equation}
with $g$ reasonably smooth, \begin{equation}\label{eqg1} g\ \
\textrm{being}\ T-\textrm{periodic}
\end{equation}
and satisfying
\begin{equation}\label{eqg2}
0<g_{min} \leq g(x)\leq g_{max}.
\end{equation}

The above problem arises in a variety of physical situations such
as Bose-Einstein condensates, nonlinear optics etc (we refer the
interested reader to \cite{alfi,toresProc,cuevas} and the
references therein for more details).


The solitary wave solutions of (\ref{eqNLS}) are given by
\begin{equation}\label{eqbs}\psi(x,t)=e^{i\lambda t} \phi(x),\end{equation} where $\phi$ is real
valued and solves
\begin{equation}\label{eqphi}
-\frac{1}{2}\phi_{xx}+\lambda \phi+g(x)\phi^3=0.
\end{equation}

A solitary wave solution $\psi$ of (\ref{eqNLS}) is called a
\emph{dark soliton} if the associated $\phi$ satisfies
\begin{equation}\label{eqasympt}
\frac{\phi(x)}{\phi_\pm(x)}\to 1\ \ \textrm{as}\ \ x\to \pm
\infty,
\end{equation}
where the functions $\phi_\pm$ are sign definite, $T$-periodic
solutions of (\ref{eqphi}).

It is worth mentioning that, in the case where $g$ is a constant,
the dark soliton can be represented explicitly and its stability
has been a topic of extensive investigations in recent years (see
\cite{gravejat} and the references therein).

If $\lambda\geq 0$, it was shown in \cite{toresProc} that the only
bounded solution of (\ref{eqphi}) is the trivial one. Therefore,
we will assume that
\begin{equation}\label{eqlambda}
\lambda<0.
\end{equation}

 Since the constant functions
 \[
\underline{\phi}=\sqrt{-\frac{\lambda}{g_{max}}}\  \ \textrm{and}
\ \ \bar{\phi}=\sqrt{-\frac{\lambda}{g_{min}}}
 \]
 are $T$-periodic lower and upper solutions respectively of (\ref{eqphi}), a well known result \cite{deCoster,opial} guarantees that
 (\ref{eqphi})
has a $T$-periodic solution $\phi_+$ such that
\begin{equation}\label{eqphiBound}
\sqrt{-\frac{\lambda}{g_{max}}}<\phi_+(x)<
\sqrt{-\frac{\lambda}{g_{min}}}.
\end{equation}
 Clearly, the function
\[
\phi_-(x)=-\phi_+(x)
\]
is also a $T$-periodic solution.

By combining several techniques from the classical theory of
ODE's, such as topological degree and free homeomorphisms (applied
to the Poincar\'{e} map), it was shown in \cite{toresProc} (see
also \cite{cuevas,toresCMP}) that (\ref{eqphi}) has a solution
that verifies (\ref{eqasympt}), under the additional assumptions
that $g$ is even and
\begin{equation}\label{eqrestrict}
g_{min}>\frac{g_{max}}{3}
\end{equation}
(in particular, the latter relation implies the uniqueness of a
$T$-periodic solution of (\ref{eqphi}) satisfying
(\ref{eqphiBound})).

In this note, we will give a simple and elementary variational
proof of this result which, in fact,  \emph{works without assuming
the last two restrictions}. At the same time, it gives some new
monotonicity properties of the obtained  solutions.

Our main observation is to adapt to this setting a remarkable
identity, due to \cite{lassouedmironescu}, which reduces
(\ref{eqphi}) to a weighted Allen-Cahn equation. This technique
has become a standard tool in the study of vortices in
inhomogeneous  equations of Ginzburg-Landau type  (see for
instance \cite{aftalionbook} and the references therein). In the
context of (\ref{eqNLS}), in the semiclassical regime, a related
idea appears in the recent paper \cite{pelimalo} which studies the
case of positive $g$ that diverges at respective infinities. The
main drawback to this approach is that, in principle, it works
only in the case of power nonlinearities. On the other hand, this
was not an issue for the approach in \cite{toresProc,toresCMP}. We
believe that this approach could also be useful  for the study of
the stability of the dark soliton in the spirit of \cite{gravejat}
and the references therein.

 More precisely, our main
result is the following. After its proof, we will remark on how to
apply this approach to the defocusing cubic-quintic nonlinear
Schr\"{o}dinger equation with a periodic potential that was
considered in \cite{toresCMP}.

\begin{thm}
Under the assumptions $g\in C(\mathbb{R})$, (\ref{eqg1}),
(\ref{eqg2}) and (\ref{eqlambda}), equation (\ref{eqphi}) admits a
solution $\phi$ that satisfies (\ref{eqasympt}). Moreover, it
holds that
\begin{equation}\label{eqpos}
\left(\frac{\phi}{\phi_+} \right)_x>0,\ \ x\in \mathbb{R},
\end{equation}
and there exists a constant $C_0>0$ such that
\begin{equation}\label{eqexp}
\left|\phi(x)-\phi_\pm(x)
\right|+\left|\left(\phi(x)-\phi_\pm(x)\right)_x \right|\leq
e^{-C_0|x|}, \ \ \pm x>0.
\end{equation}
\end{thm}
\begin{proof}
Motivated from \cite{lassouedmironescu}, we set
\[
w=\frac{\phi}{\phi_+}.
\]
Then, it follows readily that equation (\ref{eqphi}) is equivalent
to
\begin{equation}\label{eqw1}
\left(\phi_+^2w' \right)'=2g(x)\phi_+^4w(w^2-1),
\end{equation}
while the asymptotic boundary conditions (\ref{eqasympt}) are
equivalent to
\begin{equation}\label{eqw2}
w(x)\to \pm 1\ \ \textrm{as}\ \ x\to \pm \infty.
\end{equation}
It is important to observe that, thanks to (\ref{eqphiBound}), the
differential operator in (\ref{eqw1}) is nonsingular.

Equation (\ref{eqw1}) is the Euler-Lagrange equation of the energy
\[
E(w)=\int_{-\infty}^{\infty}\left\{\phi_+^2 w_x^2
+g(x)\phi_+^4(1-w^2)^2\right\}dx.
\]
Standard variational techniques  can be then employed to find a
solution of (\ref{eqw1})-(\ref{eqw2}), a heteroclinic orbit that
is, as a minimizer of $E$ in the set
\begin{equation}\label{eqset}
w\in W^{1,2}_{loc}(\mathbb{R})\ \ \textrm{and}\ w(x)\to \pm 1\
\textrm{as}\  x\to \pm \infty,
\end{equation}
see for example \cite{alessioPeriod,bonhoure,sourdisHeter} and the
references therein. Loosely speaking, conditions (\ref{eqg2}) and
(\ref{eqphiBound}) prevent the minimizing sequences from becoming
too flat, while the $T$-periodicity of $g,\phi_+$ can be used to
``pull them back'', without affecting their energy, should  their
``center of mass'' escapes to infinity.

It holds that
\begin{equation}\label{eqwleq}
|w|<1,
\end{equation}
as can be verified by observing that $\min\{w,1\}$ and $\max\{w,-1
\}$ have less or equal energy than $w$ and then applying the
strong maximum principle. We claim that
\begin{equation}\label{eqmonotWeak} w_x\geq 0.\end{equation} Indeed, let
$w(x_1)=w(x_2)$ for some $x_1<x_2$. Without loss of generality, we
may assume that $w(x_1)>0$. Since the function in
$W^{1,2}_{loc}(\mathbb{R})$ which coincides with $w$ in
$\mathbb{R}\setminus (x_1,x_2)$ and is equal to $|w|$ in
$(x_1,x_2)$ has less or equal energy than $w$, we may also assume
that $w\geq 0$ in $(x_1,x_2)$. Then, the function in
$W^{1,2}_{loc}(\mathbb{R})$ which coincides with $w$ in
$\mathbb{R}\setminus (x_1,x_2)$ and is equal to $w(x_1)$ in
$(x_1,x_2)$ has less energy than $w$ which is impossible. In fact,
we can show that $w_x>0$, which implies at once (\ref{eqpos}), by
using (\ref{eqw1}), (\ref{eqwleq}) and (\ref{eqmonotWeak}).
Finally, the convergence in (\ref{eqw2}) can be shown to be
exponentially fast, which easily implies (\ref{eqexp}), by a
standard barrier argument.

The proof of the theorem is complete.
\end{proof}

\begin{rem}
In \cite{toresCMP}, the authors studied dark soliton solutions of
the cubic-quintic NLS
\[
i\psi_t+\psi_{xx}+V(x)\psi-g_1 |\psi|^2\psi-|\psi|^4 \psi=0,
\]
where the potential $V$ is smooth,  real and $T$-periodic, while
$g_1$ is a real constant. Plugging the ansatz (\ref{eqbs}) into
the above equation yields that $\phi$ must satisfy
\begin{equation}\label{eqphi2} \phi_{xx}+\left(V(x)-\lambda \right)\phi-g_1
\phi^3-\phi^5=0.
\end{equation}
If
\begin{equation}\label{eqRest5}\lambda<\min V,\end{equation} it was shown as before, by the method of upper and lower solutions, that (\ref{eqphi2}) has a
 $T$-periodic solution $\phi_+$
such that
\[
\rho_1<\phi_+<\rho_2,
\]
where
\[
\rho_j=\frac{1}{\sqrt{2}}\sqrt{\sqrt{g_1^2+4\lambda_j^2}-g_1}\ \
\textrm{with}\  \lambda_1^2=\min V-\lambda \ \textrm{and} \
\lambda_2^2=\max V-\lambda.
\]
 As before, we let $\phi_-=-\phi_+$. If
$V$ is assumed to be even, an analogous condition to
(\ref{eqrestrict}), but considerably more complicated,  was found
that guarantees the existence of a solution to (\ref{eqphi2}) that
verifies (\ref{eqasympt}).

Our approach carries over to this problem without too much effort.
Indeed, proceeding as in the proof of our main result, we find
that the corresponding $w$ should satisfy
\[
(\phi_+^2 w')'=g_1\phi_+^4w(w^2-1)+\phi_+^6w(w^4-1)
\]
together with the asymptotic behavior (\ref{eqw2}). Now, the
associated energy is
\[
E(w)=\int_{-\infty}^{\infty}\left\{\frac{1}{2}\phi_+^2 w_x^2
+\phi_+^4\left[\frac{g_1}{4}+\frac{\phi_+^2}{6}(2+w^2)
\right](1-w^2)^2\right\}dx.
\]
If $g_1\geq 0$, the aforementioned variational arguments apply
directly to show that there exists a minimizer of the energy in
the set described  in (\ref{eqset}), without the need of imposing
any further condition other than (\ref{eqRest5}). The case $g_1<0$
is more subtle and requires a further investigation. Clearly, a
sufficient condition for the minimization argument to go through
and produce the desired solution is
\[
\frac{g_1}{4}+\frac{\rho_1^2}{3}\geq 0,
\]
which can always be made possible by choosing $-\lambda>0$
sufficiently large.
\end{rem}

 \section*{Acknowledgements} This research was supported by the
ARISTEIA (Excellence) programme ``Analysis of discrete, kinetic
and continuum models for elastic and viscoelastic response'' of
the Greek Secretariat of Research.


\begin{thebibliography}{50}
\bibitem{aftalionbook}
{\sc A. Aftalion}, Vortices in Bose Einstein Condensates, PNLDE
\textbf{67}, Birkh\"{a}user, Boston, 2006.


\bibitem{alessioPeriod}
{\sc F.G. Alessio}, and {\sc P. Montecchiari}, \emph{Layered
solutions with multiple asymptotes for non autonomous Allen-Cahn
equations in $\mathbb{R}^3$}, Calc. Var. \textbf{46} (2013),
591–-622.

\bibitem{alfi}
{\sc G. L. Alfimov}, and {\sc A. I. Avramenko}, \emph{Coding of
nonlinear states for NLS-type equations with periodic potential},
in: R. Carretero-González, J. Cuevas-Maraver, D. Frantzeskakis, N.
Karachalios, P. Kevrekidis and F. Palmero-Acebedo (Eds.),
Localized Excitations in Nonlinear Complex Systems: Current State
of the Art and Future Perspectives, in: Nonlinear Systems and
Complexity, Springer, 2014, 43--61.

\bibitem{toresProc}
{\sc J. Belmonte-Beitia}, and {\sc P. J. Torres}, \emph{Existence
of dark soliton solutions of the cubic nonlinear Schr\"{o}dinger
equation with periodic inhomogeneous nonlinearity}, J. Nonlinear
Math. Phys. \textbf{15} (2008), 65--72.

\bibitem{cuevas}
{\sc J. Belmonte-Beitia}, and {\sc J. Cuevas},  \emph{Existence of
dark solitons in a class of stationary nonlinear Schr\"{o}dinger
equations with periodically modulated nonlinearity and periodic
asymptotics}, J. Math. Phys. \textbf{52} (2011), 032702.



\bibitem{bonhoure}
{\sc D. Bonheure}, and {\sc L. Sanchez}, \emph{Heteroclinic orbits
for some classes of second and fourth order differential
equations},
 Handbook of
Differential Equations \textbf{III}, Elsevier, Amsterdam (2006),
103--202.


\bibitem{deCoster}
{\sc C. De Coster}, and {\sc P. Habets}, \emph{Upper and lower
solutions in the theory of ODE boundary value problems: classical
and recent results}, in: F. Zanolin (Ed.), Nonlinear Analysis and
Boundary Value Problems for Ordinary Differential Equations, in:
CISM-ICMS \textbf{371}, Springer-Verlag, New York, 1996, 1–-78.


\bibitem{gravejat}
{\sc Ph. Gravejat}, and {\sc D. Smets}, \emph{Asymptotic stability
of the black soliton for the Gross-Pitaevskii equation}, preprint
(2014), or hal-01002094.


\bibitem{pelimalo}
{\sc B. A. Malomed}, and {\sc D. E. Pelinovsky}, \emph{Persistence
of the Thomas-Fermi approximation for ground states of the
Gross-Pitaevskii equation supported by the nonlinear confinement},
Appl. Math. Lett. \textbf{40} (2015), 45--48.

\bibitem{lassouedmironescu}
{\sc L. Lassoued}, and  {\sc P. Mironescu}, \emph{Ginzburg-–Landau
type energy with discontinuous constraint}, {J. Anal. Math.}
\textbf{77} (1999),  1–-26.

\bibitem{opial}
{\sc Z. Opial}, \emph{Sur les int\'{e}grales born\'{e}es de l'
equation $u''=f(t,u,u')$}, Annales Polonici Math. \textbf{4}
(1958), 314--324.

\bibitem{sourdisHeter}
{\sc C. Sourdis}, \emph{The heteroclinic connection problem for
general double--well potentials}, Arxiv preprint (2013),
arXiv:1311.2856.


\bibitem{toresCMP}
 {\sc P. J. Torres},
and {\sc V. V. Konotop}, \emph{On the existence of dark solitons
in a cubic--quintic nonlinear Schr\"{o}dinger equation with a
periodic potential}, Comm. Math. Phys. \textbf{282}  (2008), 1--9.

\end{thebibliography}
\end{document}